\documentclass[a4paper,11pt]{amsart}
\usepackage{amssymb,amsfonts,amsmath,listings}

\def\C{\mathbb{C}}
\def\c2{\mathbb{C}^2}

\def\Q{\mathbb{Q}}
\def\Z{\mathbb{Z}}
\def\N{\mathbb{N}}

\def\P{\mathbb{P}}

\def\1{\bold{1}}

\def\l{\lambda}

\def\om{\omega}

\def\lbr{\lbrace}
\def\rbr{\rbrace}
\def\J{\mathcal J}
\def\P{\mathcal P}

          \newcommand\PP{{\mathbb{P}}}
           
           \newcommand\G{{\mathcal G}}

\newtheorem{lem}{Lemma}[section]
\newtheorem{pro}[lem]{Proposition}

\newtheorem{def/not}[lem]{Definition/Notations}

\newtheorem{prob}[lem]{Problem}
\newtheorem{thm}[lem]{Theorem}

\newtheorem{cor}[lem]{Corollary}
\newtheorem{rqe}[lem]{Remark}

\newtheorem{obs}[lem]{Observation}

\begin{document}

\title{Remarks on Mukai threefolds admitting $\mathbb C^{*}$ action}

\author{S\l awomir Dinew, Grzegorz Kapustka, Micha\l \ Kapustka}


\address{Department of Mathematics and Computer Science, Jagiellonian
University, Krak\'ow, Poland \newline
{\tt slawomir.dinew@im.uj.edu.pl,}\newline
Institute of Mathematics of the Polish Academy of Sciences,
ul. \'{S}niadeckich 8, P.O. Box 21, 00-956 Warszawa, Poland.\newline 
{\tt grzegorz.kapustka@uj.edu.pl,}\newline
University of Stavanger, Norway \newline
 {\tt michal.kapustka@uj.edu.pl}}

\maketitle
\begin{abstract}
We investigate geometric invariants of the one parameter family of Mukai threefolds that admit $\mathbb C^{*}$ action. In particular we find the invariant divisors in the anticanonical system, and thus establish a bound on the log canonical thresholds. Furthermore we find an explicit description of such threefolds in terms of the quartic associated to the variety-of-sum-of-powers construction. This yields that any such threefold admits an additonal symmetry which anticommutes with the $\mathbb C^{*}$ action, a fact that was previously observed near the Mukai-Umemura threefold in \cite{RST}. As a consequence the K\"ahler-Einstein manifolds in the class form an open subset in the standard topology.
\end{abstract}

\section*{Introduction} The family $V_{12}$ of compact complex threefolds with
genus 12 found by Iskovskih \cite{Is1,Is2}, (widely known as Mukai threefolds)
is the source of many interesting examples in complex geometry. In fact this
class of threefolds was missed by Fano in his classification list (see
\cite{Is1, Is2}). We refer to \cite{M} for their basic properties.

The interest of complex differential geometry towards Mukai threefolds stems
from the fact that a subclass of these appears in the famous example of Tian
(\cite{Ti})  of a compact Fano manifold with no nonzero holomorphic vector
fields yet with no K\"ahler-Einstein metric. In fact before
Tian's paper a folklore conjecture expected holomorphic vector fields to be the
only obstruction of existence of such metrics.  This example partially motivated
 a suitable notion of stability ($K$-stability defined by Tian in
\cite{Ti}) which turned out to be  the right algebro-geometric condition equivalent to the existence of K\"ahler Einstein metrics. We refer to \cite{Ti,Ti2} for the history of the
problem,  and to \cite{CDS1,CDS2,CDS3,Ti4} for its final solution. In a sense the
family of Mukai threefolds provides an excellent test polygon for such
investigations.

Below we list what is known on the K\"ahler-Einstein problem with regard to the
family $V_{12}$ so far. First of all of crucial importance is the (identity component of the) automorphism
group of such threefolds. This was an object of intensive research (see \cite{P}
and references therein). It turns out that a generic Mukai threefold has a
discrete group of automorphisms (thus the identity component is just a singleton). The remaining cases were classified by
Prokhorov (\cite{P}). If $G$ is the identity component in $Aut(X)$ then all the possibilities are  as follows:
\begin{thm}\label{P} The identity component of the automorphism group $G$ of $V_{12}$ is a singleton except in
the following cases:
\begin{itemize}
\item[1)] $ X=V_{MU}$ is the Mukai-Umemura manifold, then $G = SL(2)$;
\item[2)] $X$ is a member of the family $ V^a$, then $G=\mathbb{C}^{\ast} $;
\item[3)] $X= V^m$ then $G = \mathbb{C}^{+}$.
\end{itemize}
\end{thm}

The existence of K\"ahler-Einstein metrics was investigated in detail {\it near} the Mukai-Umemura manifold (see \cite{Ti, D1, D2, RST}). In particular the Mukai-Umemura threefold itself admits such a metric (\cite{D2}) whereas some small deformations of it do not ({\cite{Ti}).

On the other hand it is also interesting to consider the problem {\it globally} i.e. in the whole family rather than near a single member. It should be emphasized that the deformation theory near the Mukai-Umemura example is considerably different than around a generic element of the family. One of the reasons is exactly beacuse this example is the unique one having a maximal symmetry group.
Below we briefly recall what is known in the global setting.

The case of discrete group of automorphisms is reasonably understood. Indeed, coupling the results of Tian (\cite{Ti}) and Donaldson (\cite{D1}) with the theorem of Odaka \cite{O} (see also the recent paper \cite{D3}) we
get the following observation:
\begin{obs}
 The subset of $M$ parametrizing Mukai threefolds with discrete group of
automorphisms that admit  K\"ahler-Einstein metrics is Zariski open and
nonempty. However, there are special Mukai threefolds (for example class 4 and 5 in Donaldson classification, see \cite{D1}) with discrete automorphism
group that do not admit K\"ahler-Einstein metrics. 
\end{obs}

In this paper we focus on the remaining  cases i.e. when the automorphism group is
infinite.

First, as $\mathbb C^{+}$ is not a reductive group the manifold $V^m$  does not
admit a K\"ahler-Einstein metric by  the classical Matsushima theorem \cite{Ma}.

Secondly, we already mentioned above that the Mukai-Umemura manifold admits a K\"ahler-Einstein metric.

Thus the only remaining case is the family of manifolds $V^a$ with the identity component of the group of
automorphisms isomorphic to $\mathbb{C}^{\ast}$. Small deformations of the Mukai-Umemura example living in this class were investigated by Donaldson
\cite{D1} and Rollin-Simanca-Tipler \cite{RST} (they correspond to the class 3 in the list in \cite{D1}), and it turns out that {\it some} deformations of the Mukai-Umemura
example living in this class do possess a K\"ahler-Einstein metric. A natural question, raised by
Donaldson in \cite{D1}, is to classify the K\"ahler-Einstein examples in the
family.

Our goal in the present note is to study the geometry of the generic element in the family $V^a$.
We shall present three classical constructions of this manifold: as variety of sums of powers (see \cite{M1}), by birational transformations of the Fano threefold $V_5$ of degere $5$ 
(this was the original construction of Prokhorov), finally as a subset of the the Grassmannian $\G(4,7)$. We describe the relations between these constructions. In particular in Section \ref{sec4} a Macaulay 2 program is presented. This program describes the quartic being the Hilbert scheme of lines on a Fano manifold in $V_{12}$ 
constructed as a subset of $\G(4,7)$. The Hilbert scheme of lines is the covariant quartic of another quartic that gives rise through the variety of sums of powers construction the Fano threefold in $V_{12}$ back. 

First of all we discuss the "standard" approach to the problem of existence of K\"ahler-Einstein metrics by analyzing the log-canonical thresholds (equivalently: alpha invariants) and exploiting the symmetries such manifolds have. Our result, as expected, shows that the natural symmetries are not enough to solve affirmatively the K\"ahler-Einstein problem.
\begin{thm} For any element of the family $V^a$ one has
$$ lct(V^a,\C^{*})\leq \frac12.$$
\end{thm}
More importantly we were able to construct all the $\C^{*}$-invariant divisors in the linear system $|-K_{V}|$ for any $V\in V^{a}$. 

The construction by variety of sum of powers, due to Mukai (\cite{M}) was used to characterize the Mukai-Umemura threefold. It turns our that the associated plane quartic in this case is just a double conic. Our next result is the following characterization of Mukai threefolds admitting $\C^{*}$ action, which can be seen as a generalization of Mukai theorem:
\begin{pro}
Let $V$ be a member of the family $V^{a}$ of Mukai threefolds admitting (nontrivial) $\C^{*}$ action. Then the associated plane quartic in the variety of sum of powers construction is formed by two tangent conics. In the case of Mukai-Umemura threefold these conics coincide (i.e. we get a double conic), while for $V\in V^{a}$ the conics are tangent to each other at two points.
\end{pro}
We are also able to compute the quartic associated to the manifold $V^m$:
\begin{pro}The plane quartic associated to $V^m$ in the VSP construction is the sum of two conics tangent to each other at one point with equation
$$P_4(a,b,c)=(c^2-ab)^2+a^4.$$ 
\end{pro}

Our next observation is that any member of $V^a$ has an additional holomorphic involution $\iota$ which does not commute with the $\C^{*}$ action. This was observed earlier in \cite{RST} for  small deformations of the Mukai-Umemura threefold. 

Finally we consider the group $H=<\C^{*},\iota>=\mathbb Z_2\ltimes \C^{*}$, and its compact subgroup $W$ generated by the circle action and $\iota$. Note that by Bando-Mabuchi theorem \cite{BM} if a K\"ahler-Einstein metric exists then one can find a $W$-invariant one. Then an argument that we learned from \cite{PSSW} yields the following:
\begin{thm} The set of elements in $V^{a}$ admitting $W$-invariant K\"ahler-Einstein metrics is open in the Euclidean topology.
\end{thm}
\begin{rqe} In view of the recent paper of Donaldson \cite{D3} it seems reasonable that the set above is actually Zariski open. 
\end{rqe}

In our note we heavily rely on Macaulay 2 computations. The relevant scripts are incorporated into the note.

When finishing our note we learned about  Rollin-Simanca-Tipler paper \cite{RST}. There, among other things, the case of small deformations near the Mukai-Umemura exmaple was considered. In particular we would like to point out that the openness of K\"ahler-Einstein manifolds in the family $V^{a}$ follows from their arguments if one knows the existence of the symmetry $\iota$ for all elements of $V^{a}$.  On the other hand our argument seems to be a bit simpler albeit much less general (of course both approaches rely on the implicit function theorem). More importantly it seems to emphasize the group action viewpoint so we decided to include it.  

\section{Notation and basic definitions}
Denote as usual
by $\mathbb{C}^{+}$ the group of complex numbers with addition, and by
$\mathbb{C}^*$ of non-zero complex numbers
 with multiplication.

The Grassmannian of $k$-planes in $\C^n$ will be denoted by $\G(k,n)$.

Recall that a compact complex manifold X is Fano if its first Chern class $c_1(X)$
is positive or,
equivalently, its anti-canonical line bundle $\mathcal{O}_X(-K_X)$ is ample.

Given any Fano threefold the integer $g=\frac12(-K_X)^3 + 1$ is called the genus
of $X$.
If the Picard group of $X$ is generated by $O_X(-K_X)$
we call such a Fano 3-fold prime. Prime Fano threefolds were classified by
Iskovskih \cite{Is1, Is2} with respect to their genus.
In particular we have either $g\leq 10$ or $g=12$.

Given an effective $\Q$-divisor $D$ on a compact complex manifold $X$ the
multiplier ideal sheaf $\J(X,D)$ is defined as follows: if $\mu: X'\rightarrow
X$ is a log resolution then
\begin{equation}\label{multIdeal}
\J(X,D):=\mu_{\ast}\mathcal O_{X'}(K_{X'/X}-[\mu^{\ast}D]),
\end{equation}
where $K_{X'/X}$ denotes the relative canonical bundle, whereas $[A]$ stands for
the round down of the $\Q$-divisor $A$ i.e. $[A]=\sum_i[a_i]A_i$ if
$A=\sum_ia_iA_i$.

An alternative analytic way to define this sheaf (see \cite{De}) is as follows:
take a local trivializing section $\sigma$ of the line bundle $\mathcal O(-D)$
 and equip it with a (singular) metric $h$ such that in any
trivialization
$$||\zeta||_h^2=\frac{||\zeta(z)||^2}{|\sigma(z)|^2}.$$

Then the stalk of $J$ at $z$ is defined by
\begin{equation}\label{stalk}
\J(X,D)_z=\lbr f\in\mathcal O_{X,z}|\ \exists\ neighborhood\ U\ of\ z\ :
\int_U|f|^2e^{-2log|\sigma|}<\infty\rbr.
\end{equation}

The log canonical threshold associated with the pair $(X,D)$ is defined by
\begin{equation}\label{lct}
 lct(X,D):=sup\lbr\lambda\in \Q|\ the\ pair\ (X,\lambda D)\ is\ log\
canonical\rbr,
\end{equation}
where log canonicity is defined by the condition that
$\J(X,(1-\varepsilon)\lambda D)=\mathcal O_X$ for all $0<\varepsilon<1$.
Alternatively, the pair $(X,D)$ is log canonical if there is a log resolution
$\mu: X'\rightarrow X$, such that for a normal crossing divisor
$D'=\sum_iE_i+\mu_{\ast}^{-1}D$ ($E_i$ are the exceptional divisors of the
resolution) we have
$$K_{X'}+D'=\mu^{\ast}(K_X+D)+\sum_ia_iE_i,$$
with all of the coefficients $a_i$ satisfying $a_i\geq-1$.

Given a Fano manifold $X$ the associated global log canonical threshold can be
computed simply by taking the infimum over the all log canonical thresholds
associated to the divisors that are $\Q$-numerically equivalent to $-K_X$: that
is
\begin{equation}\label{linebundle}
lct(X):=inf\lbr lct(X,\frac1nD)| D\in |-nK_X|,\ n\in\N\rbr.
\end{equation}

Analogously one can define a global log canonical threshold $lct(X,L)$ on any
polarized manifold $(X,L)$.

Finally these notions have their $G$-invariant counterparts if $G$ is any
subgroup of  $Aut(X)$:

the global $G$ invariant log canonical threshold is defined by
\begin{equation}\label{Ginvariant}
lct(X,G)=sup\lbr\lambda\in\Q|\ the\ pair (X,\frac{\lambda}nD)\ is\ log\
canonical\ for\ any\ D\ in
\end{equation}
$$ any\ G\ invariant\ linear\ subsystem\ \mathcal D\subset|-nK_X|, n\in\N
\rbr.$$
It should be noted that the stalks of the sheaf $\J(X,D)$ crucially depend on the singularity of the divisor $D$ at a given point.  Roughly speaking the log canonical threshold for a single divisor measures ``how singular'' the divisor can be at a point. The computation of such thresholds is in general quite subtle but there are special cases where fairly general formulas are available. In praticular below we recall the following inequality due to Koll\'ar (\cite{K}, Propositions 8.13 and 8.14) which is particularly useful when one deals with weighted homogenous divisors:
\begin{thm}\label{Kollar} Let $f$ be a holomorphic function near $0\in\C^n$ and let $D=\lbrace f=0\rbrace$. Assign integral weights $w(x_i)$ to the variables and let $w(f)$ be the lowest weight of all monomials occurring in the Taylor expansion of $f$. Then 
$$lct(D)\leq \frac{\sum_i w(x_i)}{w(f)}.$$ 
\end{thm}

We refer to \cite{La} and \cite{K} for more background and to the recent article \cite{Che}
for many explicit computations of log canonical thresholds.

The log canonical threshold can also be computed analytically (see the appendix
to \cite{Che} by J. P. Demailly) and it equals the so-called $\alpha$ invariant
defined by Tian (\cite{Ti1}).

The fundamental fact, proven by Tian \cite{Ti1}, is that if the alpha invariant is {\it large enough} then the manifold admits a K\"ahler-Einstein metric.
\begin{thm}
 If for a compact group $G$ the log canonical threshold (equivalently the alpha invariant) satisfies
$$lct(X,G)>\frac{n}{n+1},\ \ n=dim_{\mathbb C}X,$$
then $X$ admits a  $G$-invariant K\"ahler-Einstein metric.
\end{thm}

\section{Prime Fano threefolds $V$ of genus $12$.}
In this section we recall  several equivalent constructions of Mukai threefolds:

\subsection{Constructions of $V_{12}$ I: Variety of sum of powers}
Let us start by recalling an algebraic fact proved by Kleppe \cite{Kl}, see also
\cite{Pa}:

Any homogenous polynomial of three complex variables of degree 4 can be written
as at most 7 fourth powers of linear forms. Moreover, there is only one (up to a
linear change of variables) polynomial $x_0x_1^3+x_1^2x_2^2$ that cannot be
written as six fourth powers of linear forms.

This Waring type decomposition will be crucial in the construction of the
manifolds in $V_{12}$ that we present below.

Let $F_4=0=C\subset\PP^2$ be a quartic curve such that the quartic defining $C$
cannot be written as the sum of five (or less) fourth powers of linear forms or
a sum of six fourth powers of linear forms defining a quadrangle with its
diagonals.
Then we construct the associated manifold $VSP(C,6)$ by
\begin{equation}\label{Mukai}
VSP(C,6):=cl(\lbr([l_1],\cdots,[l_6])\in \PP^{2,\ast}|\
F_4\in<l_1^4,\cdots,l_6^4>_{\C}\rbr/S_6),
\end{equation}
where $cl$ denotes the closure in the Hilbert scheme of six points in the dual
projective plane $\PP^{2,\ast}$ and the whole construction is made invariant under the permutation group $S_6$.

Mukai observed \cite{M} that each Fano threefold $V_{12}$ is isomorphic to \newline
$VSP(C,6)$ for some quartic as above, moreover
 such $C$ is unique up to isomorphism. In particular the Mukai-Umemura threefold corresponds to $C$  being a
double conic (see \cite{M}).

Our proposition classifies the family $V^a$ and the manifold $V^m$ in this
moduli:
\begin{pro} The manifolds from the family $V^a$ corresepond to \newline
$VSP(\Gamma_t,6)$, with the curve $\Gamma_t$ being a pair of conics tangent at two points. The manifold $V^m$ corresponds to a pair of conics tangent at one point. 
\end{pro}
The proof is given in Section \ref{VSP}.

\subsection{Constructions of $V_{12}$ II: Birational morphisms to $V_5$} Let $V_5$ be a smooth Fano threefold with
$Pic(V_5)=Z[H]$, such that $H^3=5$. Then $K_{V_5}=-2H$. Explicitly $V_5$ can be realized exploiting the 
embedding given by $H$ and it is defined as a generic codimension $3$
linear section of the Grassmanian $\G(2,5)\subset \PP^9$.

In such a way we can see that $V_5$ is rigid and can be described as follows:
consider the natural action of $SL_2(\mathbb C)$ on $\PP(\C_6[x,y])$. Then $V_5$ can be
seen as the closure of the orbit $U$ of $xy(x^4-y^4)$ by this
action. One can check that the manifold admits a stratification 
$$V_5=U\cup R\cup C, $$ 
where $R$ is the orbit of
$x^5y$ and $C$ is the orbit of $x^6$.

Denote by $(x_0,x_1,x_2,x_3,x_4,x_5,x_6)$ the coordinates corresponding to
\newline $(x^6,x^5y,x^4y^2,x^3y^3,x^2y^4,xy^5,y^6)$ in $\PP(\C_6[x,y])$.
Consider the divisors $F_1$ and $F_2$ defined by $x_0=0$ and $x_6=0$
respectively.

The orbits of the action of $\C^*$ on $V_5$ have four fixed points corresponding
to $x^6$, $x^5y$, $y^5x$ and $y^6$.
Let us describe the other orbits:
\begin{itemize} \item the closure of the generic orbit of $\C^*$ is a rational
curve of degree $6$ such that the closure of such a curve contains the points
$x^6$ and $y^6$. These orbits cover the set $V_5-(F_1\cup F_2)$.
\item The line passing through $x^5y$ and $y^6x$ is the closure of an orbit that
we will denote by $f$.
\item The closure of the remaining orbits are rational normal curves of degree
$5$, such that $F_2\setminus F_1$ (resp. $F_1\setminus F_2$) is covered by orbits whose closure
contains the points $x^6$ and $y^5x$ (resp. $y^6$ and $x^5y$). Moreover,
$F_1\cap F_2$ is the sum of $f$ and a rational
normal curve of degree $5$ that we denote by $t$.
\end{itemize}
Recall now the Prokhorov construction of threefolds from the family $V^a$ \cite{P1}.
Let us choose $r\subset F_1 \subset V_5$ a rational normal curve of degree $5$
different from $t$ being the closure of an orbit of the $\C^*$ action.

 Next we define the following diagram of birational maps:
first $\varphi\colon W \to V_5$ is the blow-up of $V_5$ at the curve $r$,
$Y=\varphi^{-1}(Y)$.
Let $\Sigma\subset W$ be the strict transform of $t$ and
$\rho: W \to W'$ be a flop of $\Sigma$. Finally let $\pi : W'\to V$ be a
contraction of an extremal ray.
Then $V$ is a smooth prime Fano threefold of genus $12$ with identity component of the automorphism
group isomorphic to $\C^{*}$. Furthermore the second Betti number $b_2$ equals $1$. The morphism $\pi$ contracts the divisor $F'\subset W'$, being
the strict transform of $F_1$, to a line $Z \subset V^a$ with normal bundle
$\mathcal{O}_Z(1)\oplus\mathcal{O}_Z(-2)$. Moreover, $F'$ is isomorphic to the
Hirzebruch surface $\mathbb{F}^3$
with negative section $\Sigma'$. The map $\rho^{-1}\colon W'\to W$ is the flop
of $\Sigma'$.
There is a one parameter space of degree 5 orbits as above.
The important observation stemming from this construction is the following corollary:
\begin{cor}\label{smoothV^a}
 All the threefolds arising from this construction, form a one parameter family of smooth Fano threefolds
 with automorphism group containing $\C^{\ast}$.
\end{cor}
These threefolds are parameterized by the choice of the degree $5$ orbit contained in $F_1$.
It is interesting to study the ``boundary'' of this family. In particular it is natural to ask what kind of singularities can occur there.

\subsection{Constructions of $V_{12}$ III: Zero sections of a homogenous bundle on $\G(4,7)$} \label{description zero locus of vb}
Below a third construction is presented (compare \cite{Ti2,D1}): Let $T$ denote
the universal quotient vector bundle on the Grassmanian $\G(4,7)$. Let $H$ be the (very ample) line bundle $detT$. Consider a section $s$ of the
bundle $\Lambda^2T\oplus \Lambda^2T\oplus\Lambda^2T$. It can be computed that
$c_1(3\Lambda^2E)=6H$. If the subvariety $X_s$ cut out by $s$ is smooth then by
adjunction formula
$$K_{X_s}=K_{\G(4,7)}+c_1(3\Lambda^2E)|_{X_s}=-7H+c_1(E)|_{X_s}=-H,$$
thus $X_s$ is a Fano variety.

Alternatively such a section can be constructed as follows:
take a $3$-dimensional plane $P\subset\Lambda^2\C^7$, then $X$ is defined by
$$X_P=\lbr M\in\G(4,7)|\ P|_M=0\rbr.$$

It is known (see \cite{Ti}) that the space of holomorphic vector fields
of $X_P$ can be identified with the set of matrices $A$ in $sl(7,\C)$ whose
induced action on $\Lambda^2\C^{7}$ preserves $P$.\footnote{Recall that the group $sl(7,\C)$ acts on $\Lambda^2\C^{7}$ by
$g(v_1\wedge v_2)=g(v_1)\wedge v_2+v_1\wedge g(v_2).$} Also the automorphisms of $X_P$ correspond to elements in $SL(7,\C)$ which preserve $P$.\footnote{The action is given by $g(v_1\wedge v_2)=g(v_1)\wedge g(v_2)$.}
Then the orbit of $SL(7,\C)P$ will be generically of dimension $dim [sl(7,\C)]$,
except in the special cases when the stabilizer is positive dimensional.

Below we shall describe the equations of elements of $V^a$.

 Let $W_7$ denote a vector space of dimension 7 and Let 
 $L\subset \bigwedge^2 W_7$ be a generic subspace 
of dimension 3. Then we obtain a Fano threefold of genus 12 as:
$$V^{L}_{12}=\{\alpha \in G(4,7)\subset \mathbb{P}(\bigwedge^4 W_7): \alpha \wedge \omega=0\ \ \forall \omega\in L\}=G(4,7)\cap P_L.$$
Hence the equations defining the $\mathbb{P}^{13}$ section of $G(4,7)$ are given by elements of the form 
$\omega_i\wedge \in \bigwedge^3 W_7 =(\bigwedge^4 W_7)^*$, where $(\omega_1, \omega_2, \omega_3)$ is a basis of $L$.

Any such basis vector can be written as $\omega=\sum_{i<j}a_{ij}e_i\wedge e_j$, with $e_i$,\newline $i=-3,-2,\cdots 2,3$ denoting the standard basis of $1$ forms dual to the canonical basis of $W_7$ and $a_{ij}$ are constants. Thus a generic form has $21$ components. The following technical lemma states that if $V_{12}^{L}$ is smooth then none of the basis vectors can have less than three nonzero components: 

\begin{lem} \label{slocus of dim 6}
\begin{enumerate}
 \item The singular locus of a submanifold $F_{\omega} \subset \G(4,7)$ given by the zero section of the bundle $\bigwedge^2T$
corresponding to a form \newline $\omega=e_i\wedge e_j+e_k\wedge e_l$ for some $i,j,k,l$ is of dimension $\geq 6$. 
\item The locus of a submanifold $F_{\omega} \subset \G(4,7)$ given by the zero section of the bundle $\bigwedge^2T$
corresponding to a form $\omega=e_i\wedge e_j$ is of dimension $\geq 10$.
\item If $L=span\lbrace\omega_1,\ \omega_2,\ \omega_3\rbrace$ produces a smooth Fano manifold, then none of the basis vectors $\om_i$ can be of the form as above.
\end{enumerate}
\end{lem}
\begin{proof}
1.  Let us denote by $A_{\omega}:=\{[\alpha]\in \mathbb{P}(\bigwedge^4 W_7)| \alpha\wedge \omega=0 \}$. Note that 
$F_{\omega}:=A_{\omega}\cap \G(4,W_7)$. It is not hard to check (one can also use Macaulay 2) that $\dim F_{\omega}=9$, which is the expected dimension for the zero locus of a rank 3 vector bundle.
Consider also the Schubert cycle 
$$B_{\omega}:=\{ [U]\in \G(3,W_7)| \dim (U \cap <e_i,e_j,e_k,e_l>)\geq 3\}.$$
Clearly  $B_{\omega}\subset F_{\omega}$ and $B_{\omega}$ is of dimension 6. We claim that $B_{\omega}\subset \mathrm{Sing} F_{\omega}$. Indeed, let $\alpha\in B_{\omega}$  and let $T_{[\alpha]}(\G(3,W_7))$ be the projective tangent space to $G(3,W_7)$ at $\alpha$. Then the projective tangent space to  
$$T_{[\alpha]}(F_{\omega})=T_{[\alpha]}(\G(3,W_7)) \cap A_{\omega}$$
equals
$$<\{[\beta \in \G(3,W_7)| \beta\wedge \omega=0 \text{ and } \dim [\beta]\cap[\alpha] \geq 3 ]\}>.$$
 To prove that $F_{\omega}$ is singular in $[\alpha]$, it is enough to prove that \newline $\dim (T_{[\alpha]}(\G(3,W_7)) \cap
 A_{\omega})\geq 10$.
 For this, let us consider 
$$A_{\omega,v}=\{[\alpha]\in \mathbb{P}(\bigwedge^4 W_7)| \alpha\wedge \omega\wedge v=0 \},$$
for each $v\in W_7$. 
We observe that $T_{[\alpha]}(\G(3,W_7))\subset A_{\omega,v}$ for each\newline $v\in <e_i,e_j,e_k,e_l> + <\alpha>$.
Since $$A_{\omega}=\bigcap_{v\in W_7} A_{\omega,v}$$ it follows that 
 $T_{[\alpha]}(\G(3,W_7)) \cap A_{\omega})$ has codimension at most 2 in \newline $T_{[\alpha]}(\G(3,W_7))$. Hence
$ \dim (T_{[\alpha]}(\G(3,W_7)) \cap A_{\omega})\geq 10$ and $F_{\omega}$ is singular in $[\alpha]$ proving the claim.

2. Take without loss of generality the form $e_{-3}\wedge e_{-2}$. Then clearly the 4-plane $<x_{-3},x_{-2},x_{-1},x_0>$ belongs to $F_{\omega}$ and if we look for $4$-planes in $F_{\om}$ being graphs of linear maps from $<x_{-3},x_{-2},x_{-1},x_{0}>$ to $<x_1,x_2,x_3>$ represented by a matrix
$$   \left[ {\begin{array}{cccc}
   a & b & c & d\\ 
   e & f & g & h\\
   i & j & k & l      \end{array} } \right] $$
it is easy to see that the necessary and sufficient condition is that the first two columns have to be proportional. This yields a 10-dimensional family within $F_{\omega}$.

3. It is easy to see that $F_{\omega_1}\cap F_{\omega_2}\cap F_{\omega_3}$ is a smooth Fano threefold only if these $F_{\omega_j}$ intersect transversally. In the case when one of the forms is of the type $e_i\wedge e_j$ the dimension of the intersection is higher than three. In the case when one of the forms is of the type $e_i\wedge e_j+e_k\wedge e_l$ then dimension count tells us that the dimension of the singular locus of the Fano 3-fold is at least zero i.e. it should be a finite (nonzero) set of points. Thus in this case the constructed space must be singular. 
\end{proof}

On the other hand the weights on the standard $\C^{*}$ action $t: e_j\rightarrow t^je_j$ induce weights $t^{j+k}$ on $e_j\wedge e_k$ and thus forms on which $t$ acts with weight $5,4,-4$ and $-5$ are simply multiples of  monomial forms, whereas forms on which the action has weights $3,2-2$ or $-3$ are spanned by sums of two monomial terms. In turn forms of weight $-1, 0$ or $1$ are spanned by three monomial terms. Thus the only possibility to get a smooth $\C^{*}$-invariant Fano threefold is that the weights are $(-1,0,1)$ and in each case the corresponding basis vector form has to  be of full rank i.e. no coefficient can be equal to zero.

Next, following \cite{D1}, we specify a $\C^{*}$-invariant 3 plane $L$ by making use of a $\mathbb{C}^*$ 
action on $W_7$. By Lemma \ref{slocus of dim 6} in order to obtain a smooth Fano 3-fold, we have to choose $L$ to be spanned by three 2-forms 
 $(\omega_{-1}, \omega_{0}, \omega_{1})$ on which $\mathbb{C}^*$ acts by weights $ (-1, 0, 1)$. 
 As explained in \cite{D1},
 such a triple of $2$ forms can be written as
\begin{align*}\omega_{1} &= u_{01}e_0\wedge e_1+u_{-12}e_{-1}\wedge e_2+ u_{-23}e_{-2}\wedge e_3 \\
   \omega_{0} &= v_{-11}e_{-1}\wedge e_1+v_{-22}e_{-2}\wedge e_2+ v_{-33}e_{-3}\wedge e_3 \\
 \omega_{-1} &= w_{-10}e_{-1}\wedge e_0+w_{-21}e_{-2}\wedge e_1+ w_{-32}e_{-3}\wedge e_2,  
\end{align*}
for some constants $u_{ij},v_{ij},w_{ij}\in \C\setminus\lbrace0\rbrace$. Of course the constants are not uniquely defined, as rescaling the coordinates $x_i$, as well as the whole 2-forms yield the same invariant space $L$.  Exploiting the rescalings of the coordinates $x_i$ by factors $\lambda_i$ given by
\begin{align*}
&\lambda_{-3}=\sqrt{\frac{u_{01}}{w_{-10}w_{-32}^2}\sqrt{\frac{v_{-22}u_{-12}^3}{v_{-11}w_{-21}}}}\\
&\lambda_{-2}=\sqrt{\frac{u_{01}}{w_{-10}}\sqrt{\frac{v_{-11}u_{-12}}{v_{-22}w_{-21}^3}}}\\
&\lambda_{-1}=\sqrt{\frac{u_{01}}{w_{-10}}\sqrt{\frac{v_{-22}}{v_{-11}w_{-21}u_{-12}}}}\\
&\lambda_{-1}=\sqrt{\frac{u_{01}}{w_{-10}}\sqrt{\frac{v_{-22}}{v_{-11}w_{-21}u_{-12}}}}\\
&\lambda_{0}=\sqrt{\frac{1}{u_{01}w_{-10}}\sqrt{\frac{v_{-22}}{v_{-11}w_{-21}u_{-12}}}}\\
&\lambda_{1}=\sqrt{\frac{w_{-10}}{u_{01}}\sqrt{\frac{v_{-22}}{v_{-11}w_{-21}u_{-12}}}}\\
&\lambda_{2}=\sqrt{\frac{w_{-10}}{u_{01}}\sqrt{\frac{v_{-11}w_{-21}}{v_{-22}u_{-12}^3}}}\\
&\lambda_{3}=\sqrt{\frac{w_{-10}}{u_{01}u_{-23}^2}\sqrt{\frac{v_{-22}w_{-21}^3}{v_{-11}u_{-12}}}}
\end{align*}

followed by rescaling of $\omega_0$ by
$$\sqrt{\frac{u_{-12}w_{-21}}{v_{-11}v_{-22}}}$$ 
we end up with the following basis of $L$: 
 \begin{align*}
 \omega^{L}_{1} &= e_0\wedge e_1+e_{-1}\wedge e_2+e_{-2}\wedge e_3 \\
   \omega^{L}_{0} &= e_{-1}\wedge e_1+e_{-2}\wedge e_2+ \tau e_{-3}\wedge e_3 \\
 \omega^{L}_{-1} &= e_{-1}\wedge e_0+e_{-2}\wedge e_1+ e_{-3}\wedge e_2, 
 \end{align*}
where $\tau=\frac{u_{-12}v_{-33}w_{-21}}{u_{-23}v_{-11}w_{-32}}$ (a quantity equivalent to the modulus defined by Donaldson in \cite{D1}). Exactly like in \cite{D1} this $\tau$ parametrizes the family of $3$ planes $L_{\tau}$. In particular we obtain 
 the Mukai-Umemura case for $\tau=1$. It should be emphasized however, that this parametrization is {\it noneffective} i.e. there might be
 different $\tau$'s corresponding to isomorphic manifolds. In particular Tian's choice of coordinates for the Mukai-Umemura example
 \cite{Ti1} yields $\tau=5$.

 Our aim is now to describe the equations defining $V_{12}^{L_{\tau}}\subset \PP(\bigwedge^4 W_7)$.
 
Starting from a one parameter family $L_{\tau}$ as above we describe the linear space $L12\subset \bigwedge^4 W_7$ describing the minimal generators of the Fano threefold $V_{12}^{L_d}= \G(4,7)\cap L12\subset \PP(\bigwedge^4 W_7)$.
This is done by exploiting the following program that runs in Macaulay 2.
In particular \emph{mingens L12} above finds the linear equations describing $L12 \subset \bigwedge^4 W_7$
and \emph{mingens V12} finds the equations defining $V_{12}^{L_{\tau}} \subset \PP(\bigwedge^4 W_7)$ and
 \emph{mingens V12inPL12} finds the equations of $V_{12}^{L_{\tau}} \subset \PP(L12)$.
Note that this generalizes the results obtained in \cite{F} concerning the equations of the Mukai-Umemura threefold. In the following we put $\tau=-d$

\begin{lstlisting}
S=frac (QQ[d])
R=S[x_0..x_6, SkewCommutative=>true]
u=x_0*x_5+x_1*x_4+x_2*x_3
w=-d*x_0*x_6+x_1*x_5+x_2*x_4
v=x_1*x_6+x_2*x_5+x_3*x_4
MING3=mingens ((ideal(vars(R)))^3)
MING4=rsort mingens ((ideal(vars(R)))^4)
NORM34=substitute ((transpose MING4)*(MING3),
{x_0=>1, x_1=>1,x_2=>1,x_3=>1,x_4=>1,x_5=>1,x_6=>1})
MU=matrix {apply(7, i->(coefficients(u*x_i, 
Monomials=>MING3))_1) };
MV=matrix {apply(7, i->(coefficients(v*x_i,
 Monomials=>MING3))_1) };
MW=matrix {apply(7, i->(coefficients(w*x_i,
 Monomials=>MING3))_1) };
RG=S[p0123,p0124,p0134,p0234,p1234,p0125,p0135,
p0235,p1235,p0145,p0245,p1245,p0345,p1345,p2345,
p0126,p0136,p0236,p1236,p0146,p0246,p1246,p0346,
p1346,p2346,p0156,p0256,p1256,p0356,p1356,p2356,
p0456,p1456,p2456,p3456];
 KOND=(map(RG,R))(NORM34*(MU|MV|MW))
 GG=Grassmannian(3,6,RG);
 L12=ideal (vars(RG)*KOND);
 V12=GG+L12;
P13=substitute(V22,{p2356=>-p1456,
p2346=>p0456,p1346=>d*p2345,
p1246=>d*p1345,p0156=>p1236,
p0236=>d*p1235,p0146=>-d*p1235,
p0346=>d*p1345,p0256=>-d*p1345,
p0345=>p1236,p0245=>d*p1235,
p0136=>-d*p1234,p0235=>d*p1234,
p0234=>p0126,p0135=>-p0126,
p0134=>-p0125,p1356=>-p0456,
p0356=>-d*p2345, p0135=>-p0126,
p0145=>-(d+1)*p1234,p1256=>-(d+1)*p2345,
p0246=>d*p1245-p1236})
T=S[p0123,p0124,p0125,p0126,p1456,p0456,p1345,p1236,
p1235,p1234,p2345,p1245,p2456,p3456]
V12inPL12=(map(TT,RG)) P13;
degree V12inPL12
dim V12inPL12
\end{lstlisting}

\section{VSP(C,6)}\label{VSP} Recall that in the previous section a construction of  $V_{12}$
using an appropriate plane quartic was explained. Thus  the ``moduli'' space of  $V_{12}$ is birational to the moduli space of plane quartics.  
We will be interested in the one parameter family parametrized by two tangent conics i.e. 
$$\Gamma_t=\lbrace[a:b:c]\in\PP^2\ |\ (a^2+bc)(ta^2+bc)=0\rbrace.$$
  Below we prove that $\Gamma_t$ parametrize the threefolds $V^a$. 

Let $C$ be the Hilbert scheme of lines on $V_{12}$. 

Suppose that there is a quartic curve $\Gamma$ such that $C$ is the covariant quartic of $\Gamma$. 
Then (see \cite[\S 5]{M1}) we have $V_{12}=VSP(\Gamma, 6)$.
We shall need the following lemma:
\begin{lem}\label{lem-conic}
The covariant quartic of a reducible quartic being the sum of two tangent conics is also a quartic being the sum two tangent conics. The number of tangencies is preserved.
\end{lem}
\begin{proof}
 This follows from a straitforward computation using \cite[\S 8]{DK}.
\end{proof}

\begin{pro} All Fano threefolds $V^a_{t}$ can be constructed as
$VSP(\Gamma_t,6)$, where $\Gamma_t$ are two tangent conics.
\end{pro}
\begin{proof} First observe that the quartics $\Gamma_t$ admit $\C^*$ as a subgroup
of the automorphism group. Indeed
\begin{equation}
 \C^{*}\ni\lambda\mapsto\theta_{\lambda}\colon (a,b,c)\to (a,\lambda b,\lambda^{-1}c)
\end{equation}
fixes $\Gamma_t$.

The second step is to show that the above automorphisms of $\Gamma_t$ 
induce an one parameter family of automorphisms of
$VSP(\Gamma_t,6)$. Indeed, if $\theta$ is such an automorphism and
$$(a^2+bc)(ta^2+bc)=l_1^4+\dots+l^4_{6},$$ then $(a^2+bc)(ta^2+bc)=l_1^4(\theta)+\dots+l_6^4(\theta)$, where $l_j(\theta)$ denotes the action of $\theta$ on the linear form $l_j$.
We claim that the induced automorphism is not trivial for $\theta\neq 1$.
Indeed, suppose that it is trivial for a generic $\lambda$ (by Theorem \ref{P} it suffices to check for generic $\lambda$). Then $l_1(\theta_{\lambda})=e l_i$ for some $1\leq i\leq 5$ and some $e\in \C$. It follows that if $l_1=xa+yb+zc$ then two of the numbers $x,y,z$ are zero (the same holds for $l_2,\dots,l_5$).
However, we see that $(a^2+bc)(ta^2+bc)$ is not a linear combination of $a^4,b^4,c^4$ the claim follows.
\begin{rqe} A more tedious analysis of the various cases easily shows that the action is nontrivial unless $\lambda^2=1,\lambda^3=1$ or $\lambda^6=1$. Each of these cases can be ruled out by hand exploiting careful computations.  
\end{rqe}

 We know from \cite{P1} that the manifolds $V_t^a$ form a one parameter family, like the
family of two tangent conics.
In order to conclude we should show that the threefolds $VSP(\Gamma_t,6)$ are smooth.
It follows from \cite[\S 5.3]{D1} that the one parameter family $V^a_t$ is the same as the one parameter family $V_{\tau}^a$ parameterized by $\tau$ in \cite{D1}. Thus the Hilbert scheme of lines on $V_{\tau}^a$ are two tangent conics. 
We conclude with Lemma \ref{lem-conic} since each pair of conics can be obtained as a covariant quartic of a pair of conics.% sprawdzic??
 \end{proof}
The two components of the quartic above parameterize two families of lines of $V^a_t$.
\begin{cor}
 There are two divisors spanned by lines on a generic Fano threefold $V^a_t$.
\end{cor}
Note that from \cite{M} the intersection of the conics parameterizing the Hilbert scheme of lines 
corresponds to lines with special normal bundle $(1,-2)$.
From now on we will identify the families $V^a$, $V^a_t$, $V^{a}_{\tau}$\newline  and $VSP(\Gamma_t,6)$.

We finish this section by analyzing the manifold $V^m$. In order to find a $\C^{+}$-invariant $3$-plane in $\bigwedge \C^7$, we consider the standard $\C^{+}$ action on one forms given by the matrix

\begin{equation}
\C^{+}\ni t\rightarrow\left[\begin{array}{ccccccc}
1&6t&15t^2&20t^3&15t^4&6t^5&t^6\\
0&1&5t&10t^2&10t^3&5t^4&t^5\\
0&0&1&4t&6t^2&4t^3&t^4\\
0&0&0&1&3t&3t^2&t^3\\
0&0&0&0&1&2t&t^2\\
0&0&0&0&0&1&t\\
0&0&0&0&0&0&1
\end{array}\right]
\end{equation}
By direct computation the following $3$-forms yield a $\C^{+}$-invariant 3-plane

\begin{align*}
&\omega_1=e_1\wedge e_6-5e_2\wedge e_5+10e_3\wedge e_4+e_5\wedge e_6\\
&\omega_2= e_0\wedge e_6-4e_1\wedge e_5+5e_2\wedge e_4+(2e_3\wedge e_6-6e_4\wedge e_5)\\
&+e_4\wedge e_6+2e_5\wedge e_6\\
&\omega_3= e_0\wedge e_5-5e_1\wedge e_4+10e_2\wedge e_3+(e_2\wedge e_6-2e_3\wedge e_5)\\
&+(e_3\wedge e_6-2e_4\wedge e_5)
 +e_4\wedge e_6+ e_5\wedge e_6.
\end{align*}
The corresponding variety turns out to be a smooth Fano threefold. This example is an element of Tian's family of deformations (see \cite{Ti}). The corresponding plane quartic in the VSP construction can be computed in Macaulay 2: it reads
$$P_4(a,b,c)=32a^4  + 16a^2 b^2  - 8abc^2  + c^4=(c^2-4ab)^2+32a^4.$$

It is obvious that this $P_4$ is a deformation of the double conic and it is a sum of two conics intersecting at one point.

\section{Special divisors} Similarly as in \cite{F} our  goal is to describe the singularities of invariant divisors on $V_{12}^{L_d}$. Let us first describe divisors in the system $|-K_{V_{12}^{L_m}}|$ which are invariant under
the $\mathbb{C}^*$ action. Recall that $V_{12}^{L_d}$ is obtained as the intersection $\G(4,7)\cap \mathbb{P}(L)$, where 
$\mathbb{P}(L)\simeq \mathbb{P}^{13}$. Observe now that each invariant divisor in  $|-K_{V_{12}^{L_d}}|$ can be obtained as 
the intersection in $\mathbb{P}(\bigwedge^4 W_7)$ of  $\G(4,7)\cap \mathbb{P}(L)$ with an invariant hyperplane in 
$\mathbb{P}(\bigwedge^4 W_7)$. The latter hyperplanes correspond to fixed points of the action of $\mathbb{C}^*$ on  
$\mathbb{P}(\bigwedge^3 W_7)$. The set of fixed points of the latter action is a sum of 13 linear spaces $P_k$, where for $k\in\{-6,\dots 6\}$
the space $P_k$ consists of all 3-forms on which $\mathbb{C}^*$ acts with weight $k$. One checks easily that restrictions of the corresponding 
hyperplanes to $\mathbb{P}(L)$ give rise to 12 fixed divisors corresponding to $k\neq 0$ and a line of fixed divisors for $k=0$.
\begin{pro} The set of divisors fixed by the $\mathbb{C}^*$ action consists of 12 points and a line. 
\end{pro}
\begin{rqe} On the other hand we have the following invariant divisors
on $V^a$:
First recall from \cite{F1} that on a generic $V^a$ we have two special lines $l_1,l_2$ having the
normal bundle $(-1,2)$.
\begin{itemize}
\item the divisors spread by lines on $V^a$: we have two components each
belonging to $|-K_{V_{12}^a}|$. There are two families of lines;
\item the two divisors spread by conics on $V^a$ cutting the line $l_i$ for
$i=1,2$. We denote these divisors by $C^1, C^2$.
 \end{itemize}
\end{rqe}

We denote by $p_{ijkl}$ for $0\leq i,j,k,l\leq 6$ the natural coordinates of $\bigwedge^4 W_7$. 
Our aim is to understand the singularities of the invariant divisors. To this end we compute some affine parts (containing the most singular points on these divisors) of the twelve invariant divisors and the one parameter family on $V_{12}^{L_{d}}$.
These affine parts can be described by an elimination of variables as hypersurfaces with the following equations:
\begin{enumerate}
\item the divisor $ p_{0123} = 0$ \\ with equation $0=p_{1456}^5 p_{2456}^2+2p_{1356}p_{1456}^3 p_{2456}^3+p_{1356}^2 p_{1456} p_{2456}^4-1/(d+1)p_{1456}^6+(
      2d^3+2d^2+d-2)/(d+1)p_{1356}p_{1456}^4p_{2456}+(2d^3-1)/(d+1)p_{1356}^2p_{1456}^2p_{2456}^2+(-
      2d^2-d)/(d+1)p_{1356}^3p_{2456}^3+(d^4+d^2+2d)p_{1356}^2p_{1456}^3+(-2d^3+2d)p_{1356}^     
      3p_{1456}p_{2456}+(-d^4-2d^3-d^2)p_{1356}^4 $
\item the divisor $p_{0124}=0$ \\with equation $0= p_{1456}^4p_{2456}^3+2p_{1356}p_{1456}^2p_{2456}^4+p_{1356}^2p_{2456}^5\newline +(d-1)p_{1456}^5p_{2456}+(
      2d^2+4d-2)p_{1356}p_{1456}^3p_{2456}^2\newline 
+(2d^2+3d-1)p_{1356}^2p_{1456}p_{2456}^3+(d^3-d^2-d)
      p_{1356}p_{1456}^4+(d^4+4d^3+d^2)p_{1356}^2p_{1456}^2p_{2456}+(d^3+d^2+d)p_{1356}^3p_{2456}^2+(d^5+
      3d^4+3d^3+d^2)p_{1356}^3p_{1456} $
  \item the divisor $p_{0125}=0$ \\with equation
$0=p_{1456}^4p_{2456}^2+2p_{1356}p_{1456}^2p_{2456}^3+p_{1356}^2p_{2456}^4-p_{1456}^5+(2d^2+2d-2)
      p_{1356}p_{1456}^3p_{2456}+(2d^2+2d-1)p_{1356}^2p_{1456}p_{2456}^2+(d^4+2d^3+2d^2+d)p_{1356}^2p_{1456}
      ^2+(d^2+d)p_{1356}^3p_{2456} $
\item the divisor $p_{0135}=0$ \\with equation $0=p_{1456}^4p_{2456}+2p_{1356}p_{1456}^2p_{2456}^2+p_{1356}^2p_{2456}^3\newline +(d^2-1)p_{1356}p_{1456}^3+(d^2-
      1)p_{1356}^2p_{1456}p_{2456}+(d^3+2d^2+d)p_{1356}^3 $
\item  the divisor $p_{0136}=0$ \\with equation $0= p_{0125}^2 + p_{0124}p_{0135} $ 
\item the divisor $ p_{0146}=0$ \\with equation $0=p_{1456}^3 p_{2456}+p_{1356}p_{1456}p_{2456}^2+d^2p_{1356}p_{1456}^2\newline -dp_{1356}^2p_{2456} $
\item the divisor $p_{0256}=0$ \\with equation $0 =p_{1456}^2 p_{2456}+p_{1356} p_{2456}^2+(d^2+d)p_{1356} p_{1456} $
\item the divisor $p_{0356}=0$ \\with equation $p_{1456}^2+p_{1356}p_{2456} $

\item the divisor $p_{1356} =0$ \\smooth in the affine part $p_{3456}=1$ for $d\neq -1$
\item the divisor $ p_{1456} =0$ \\smooth in the affine part $p_{3456}=1$ for $d\neq -1 $

\item the divisor $ p_{2456}=0$ \\smooth in the affine part $p_{3456}=1$ for $d\neq -1$

\item the divisor $ p_{3456} =0$\\ with equation $0= p_{0124}^2p_{0125}^5+2p_{0124}^3p_{0125}^3p_{0135}+p_{0124}^4p_{0125}p_{0135}^2\newline -1/(d+1)p_{0125}^6+(
     2d^3+2d^2+d-2)/(d+1)p_{0124}p_{0125}^4p_{0135}\newline +(2d^3-1)/(d+1)p_{0124}^2p_{0125}^2p_{0135}^2+(- 
      2d^2-d)/(d+1)p_{0124}^3p_{0135}^3\newline +(d^4+d^2+2d)p_{0125}^3p_{0135}^2+(-2d^3+2d)
      p_{0124}p_{0125}p_{0135}^3\newline +(-d^4-2d^3-d^2)p_{0135}^4 $

\item one parameter family of invariant divisors parameterized by \newline $b$: $p_{0156}=b*p_{0246}$ with equations $0=p_{0124}^2p_{0125}^2+p_{0124}^3p_{0135}+1/bp_{0125}^3+(d^2b+db+1)/b*p_{0124}p_{0125}p_{0135}+(d^3b+d^2b-d^2-d)/b*p_{0135}^2 $
\end{enumerate}
It is interesting that all these divisors are weighted homogeneous. The next proposition  summarizes what Koll\'ar's inequality \ref{Kollar} yields for each of them:
\begin{pro}\begin{enumerate}
\item  The first divisor is weighted homogenous with \\ weights on $p_{1456},p_{2456},p_{1356}$ given by
 $(3,2,1)$,respectively. Then the polynomial becomes homogenous of degree $12$. In particular \newline
$lct(D_1)\leq \frac{2+3+1}{12}=\frac12$;
\item The second divisor is weighted homogenous with weights on \\$p_{1456},p_{2456},p_{1356}$ given by
 $(3,2,1)$,respectively. Then the polynomial becomes homogenous of degree $11$. In particular \newline
$lct(D_2)\leq \frac{2+3+1}{11}=\frac{6}{11}$;
\item The third divisor is weighted homogenous with weights on \\$p_{1456},p_{2456},p_{1356}$ given by
 $(3,2,1)$,respectively. Then the polynomial becomes homogenous of degree $10$. In particular \newline
$lct(D_3)\leq \frac{2+3+1}{10}=\frac{3}{4}$;
\item The fourth divisor is weighted homogenous with weights on \\$p_{1456},p_{2456},p_{1356}$ given by
 $(3,2,1)$,respectively. Then the polynomial becomes homogenous of degree $9$. In particular 
$lct(D_4)\leq \frac{2+3+1}{9}=\frac23$;
\item  The fifth divisor is homogenous. The polynomial is of degree $2$. In particular 
$lct(D_5)\leq \frac{1+1+1}{2}=\frac32$. Of course in this case we get a nodal singularity and much better estimate can be given by different neans;
\item The sixth divisor is weighted homogenous with weights on \\$p_{1456},p_{2456},p_{1356}$ given by
 $(3,2,1)$,respectively. Then the polynomial becomes homogenous of degree $7$. In particular 
$lct(D_6)\leq \frac{2+3+1}{7}=\frac67$;
\item The seventh divisor is weighted homogenous with weights on \\$p_{1356},p_{1456},p_{2456}$ given by
 $(3,2,1)$,respectively. Then the polynomial becomes homogenous of degree $5$. In particular 
$lct(D_7)\leq \frac{2+3+1}{5}=\frac65$;
\item The eighth divisor is homogenous. The polynomial is of degree $2$. In particular 
$lct(D_8)\leq \frac{1+1+1}{2}=\frac32$. As in the fifth case we get a nodal singularity; 
\item The twelfth divisor is weighted homogenous with weights on \\$p_{0124},p_{0125},p_{0135}$ given by
 $(1,2,3)$,respectively. Then the \\ polynomial becomes homogenous of degree $12$. In particular \\
$lct(D_{12})\leq \frac{2+3+1}{12}=\frac12$;
\item The thirteenth family of divisors are weighted homogenous with \\ weights on $p_{0124},p_{0125},p_{0135}$ given by
 $(1,2,3)$,respectively. Then the polynomial becomes homogenous of degree $6$. In particular \\
$lct(D_{12})\leq \frac{2+3+1}{6}=1$;
           \end{enumerate} 
\end{pro}
\begin{rqe}
 In \cite{F} part of this computation was performed for the Mukai-Umemura threefold. In particular the first and the twelfth divisors
 were
 computed explicitly.
\end{rqe}

An immediate consequence is the following upper bound for the log canonical threshold:
\begin{pro} For any $V\in V^a$ the log canonical threshold satisfies 
$$lct(V,\C^{*})\leq \frac{1}{2}.$$
\end{pro}

\section{The group of automorphisms of elements of $V^a$ and the openness of K\"ahler-Einstein examples}

The aim of this section is to study more precisely the automorphism group of elements $V\in V^{a}$ and to use the symmetries in the K\"ahler-Einstein deformation problem.

 \begin{pro} Every Fano threefold $V^a_{12}$ admits an additional
automorphism being an involution $\iota$. Furthermore if $\theta_{\lambda},\ \lambda\in\C^{*}$ denotes an element of the $\C^{*}$ action then we have the identity $\iota\circ\theta_{\lambda}=\theta_{\lambda}^{-1}\circ\iota$ i.e. the involution anti-commutes with the action.
 \end{pro}
\begin{proof} It is enough to observe that two mutually tangent conics form a
symmetric quartic. Recall that modulo linear change of coordinates such two tangent conics are described by
$$\Gamma_t=\lbrace[a:b:c]\in\PP^2 |\ (ta^2+bc)(a^2+bc)=0\rbrace.$$
 Observe that there is a natural symmetry of $\PP^2$ exchanging the two tangent points given by
\begin{equation}\label{iota}
\iota:\PP^2\ni[a:b:c]\mapsto[a:c:b]\in\PP^2 
\end{equation}
which preserves $\Gamma_t$ and hence induces an automorphism on $VSP(\Gamma_t,6)$. The identity then follows trivially.
\end{proof}
Note that an automorphism of $V^a $ induces an automorphism on the Hilbert scheme of lines.

The existence of $\iota$ has the following consequence:
\begin{thm} The set of all $V$'s in the family $V^a$ which admits K\"ahler-Einstein metrics is open in the Euclidean topology.
\end{thm}
\begin{proof}The openness near the Mukai-Umemura threefold follows from \cite{D1} or \cite{RST}. Consider now a K\"ahler-Einstein element $V_{t_0}$ and a smooth family$V_t$  of deformations within $V^a$.  For each such $V_t$ we consider the space $\mathcal C_W^{\infty}$ of smooth functions invariant under the group
 $$W=\Z_2\ltimes S^1.$$
 Note that by Bando-Mabuchi theorem \cite{BM} if a K\"ahler-Einstein metric exists one can also find a $W$-invariant one.

 Choose a smooth family of $W$ invariant forms $\omega_t$
 reperesenting the first Chern classes of $V_t$ and let $f_t$ be the corresponding Ricci potential i.e. 
$$Ric(\omega_t)-\omega_t=i\partial\bar{\partial}f_t,\  \ \int_{V_t}e^{f_t}\omega_t^3=c_1(V_t)^3.$$

As is well known existence of $W$-invariant K\"ahler-Einstein metric is equivalent to the existence of an $W$-invariant function $u_t$ solving the problem
\begin{equation}
*_t\begin{cases}(\omega_t+i\partial\bar{\partial}u_t)^3=e^{f_t-u_t}\omega_t^3,\\
\omega_t+i\partial\bar{\partial}u_t>0,\ u_t\in C_W^{\infty}(V_t).
\end{cases}
\end{equation}
In order to prove the openness it suffices to apply the implicit function theorem. Indeed by assumption the problem $*_{t_0}$ is solvable. The linearized operator at $t_0$ is $\Delta_{t_0}+id$, where $\Delta_{t_0}$ denotes the Laplacian with respect to the $W$-invariant K\"ahler-Einstein metric $\omega_{t_0}+i\partial\bar{\partial}u_{t_0}$. This operator is not invertible as it has a nontrivial one dimensional kernel (corresponding to the holomorphic vector field induced by the $\C^{*}$ action). The operator however is invertible when restricted to $W$-invariant functions (see \cite{PSSW}). Indeed as in Theorem 2  in \cite{PSSW} the group $W$ is a closed subgroup stabilizing $\omega_{t_0}+i\partial\bar{\partial}u_{t_0}$, whose centralizer in $Stab(\omega_{t_0}+i\partial\bar{\partial}u_{t_0})$ is finite by the anticommutativity of $\iota$.  Note that it in \cite{PSSW} it is assumed that $Stab(\omega_{KE})$ is a subgroup of the identity component of the automorphisms (which is not the case in our setting) but 
the proof applies verbatim without this requirement.\footnote{The first author wishes to thank J. Sturm for pointing this out.} Then Theorem 2 in \cite{PSSW} implies that the space of $W$-invariant functions is orthogonal to the kernel of $\Delta_{t_0}+id$ with respect to the $L^2(V_{t_{0}},(\omega_{t_0}+i\partial\bar{\partial}u_{t_0})^3)$. Thus $\Delta_{t_0}+id$ is an invertible operator from $C_W^{2,\alpha}$ to $C_W^{\alpha}$ which yields $C^{2,\alpha}$ smooth solution of $*_t$ for $t$ sufficiently close to $t_0$. Further $C^{\infty}$ regularity  of the solutions $u_t$ is standard.

\end{proof}
\section{The Hilbert scheme of lines on $V_{12}\subset \PP^{13}$}\label{sec4}
The aim of this section is to give an algorithm to find the Hilbert scheme of lines on any Fano 3-fold $V_{12}^{L_{\lambda}}$. We present an implementation of the algorithm in Macaulay 2.
We proceed as follows. Each line  $l\subset V_{12}^{L_{\lambda}}$ is a line in the Grassmannian $\G(4, W_7)$. Hence, there exist 
subspaces $K^l_3\subset T^l_5 \subset W_7$ of dimension 3 and 5 respectively, such that $l=F(K^l_3,4, T^l_5)\subset \G(4, W_7)$, where $F$ denotes the flag variety. Let now 
$\gamma_l\in \bigwedge^3 W_7$ be the simple  3-form associated to $K^l_3$. Consider the map:
$$\varphi_l: L_{\lambda}\ni \omega\mapsto \omega\wedge \gamma_l \in \bigwedge^5 W_7.$$
We claim that the image $\varphi_l$ is a dimension 1 subspace $\bigwedge^5 W_7$ consisting of simple forms. 
Indeed by the description of $l$ we infer that  $\omega\wedge \gamma_l$ is anihilated by $T^l_5$.
Then for each non-zero element $\theta\in \varphi_l( L_{\lambda})$ by reduction with $\gamma_l$, we associate an element
$\tilde{\theta} \in \wedge^2 (W_7/K^l_3)$. The latter satisfies $\tilde{\theta}\wedge \alpha=0$ for each 
$\alpha \in T^l_5/K^l_3$. Since $\dim T^l_5/K^l_3=2$ and $\dim W_7/K^l_3 =4$, the latter implies $\tilde{\theta}$ is a simple form and 
$\tilde{\theta}_1$ and $\tilde{\theta}_2$ are proportional for any two nonzero $\theta_1, \theta_2 \in \varphi_l( L^{\lambda})$.
This proves the claim.
We now have a map $$\Phi: \mathrm{Hilb}_{t+1}(V_{12}^{L_{\lambda}})\ni l \mapsto [\ker \varphi_l] \in \mathbb{P}((L^{\lambda})^{*}).$$

To perform the computation we use the following Macaulay 2 script.
The program below finds the quartic being the Hilbert scheme of lines on $V_{12}^L$
where $L$ is generated by $u,v,w$.
To make the program work we have to fix $u,v,w$ i.e. specify the coefficient $d$.

\begin{lstlisting}
S=QQ[d]
R=S[x_0..x_6, SkewCommutative=>true]
u=x_1*x_6+x_2*x_5+ x_3*x_4
w=-d*x_0*x_6+x_1*x_5+x_2*x_4
v=x_0*x_5+x_1*x_4+x_2*x_3
IL=(gens (intersect(kernel (-u+v),kernel (v-w),  
kernel (u-w))))_{21}
KO=IL*matrix{{u,v,w}}

MAT=((coefficients (KO_0_0))_1)
|((coefficients (KO_1_0))_1)
|((coefficients (KO_2_0))_1)
MING2=mingens (ideal(vars(R))*ideal(vars(R)))
MING5=(coefficients (KO))_0
NORM=substitute ((transpose MING5)*(MING2), 
{x_0=>1, x_1=>1,x_2=>1,x_3=>1,x_4=>1,x_5=>1,x_6=>1})
UL=MING2*NORM*MAT
ID2=(map(S,R))(ideal((coefficients ((UL_0_0)^2))_1)+ 
ideal((coefficients ((UL_1_0)^2))_1)+ 
 ideal((coefficients ((UL_2_0)^2))_1));
ID=saturate ID2
\end{lstlisting}

\section{Concluding remarks}
Fano threefolds $V$ in the family $V^a$ are examples of $T$-varieties (see \cite{S}). K\"{a}hler-Einstein metrics on symmetric $T$-
 varieties were studied in \cite{S} (our examples are not symmetric and have complexity two). In this context it is natural to study the
 rational $\C^{\ast}$ quotient $V\to Y$ where $Y$ is a surface and to describe the possible targets $Y$.
 It would be also very interesting to understand $\C^{\ast}$--invariant test configurations for elements of $V_a$.
 Then the central fiber of such a configuration is a $T$-variety with complexity one, which might be  singular.

Given the importance of the additional symmetry $\iota$ the following problem seems natural:
\begin{prob} For $V\in V^a$ find the log canonical threshold $lct(V,\Z_2\ltimes S^1)$.
\end{prob}

{\bf Acknowledgments.} The first and second named authors were supported by NCN grant 2013/08/A/ST1/00312 and the third by the grant Iuventus plus 0301/IP3/2015/73 "Teoria reprezentacji oraz wlasno\'sci rozmaito\'sci siecznych". We wish to express our gratitude to J. Sturm and Y. Prokhorov for helpful discussions.

\end{document}